\documentclass[a4paper,12pt]{article}

\usepackage{amsmath,amssymb,amsthm,mathtools}
\usepackage{geometry}
\usepackage{graphicx}
\usepackage{subcaption}
\usepackage{tikz}
\usepackage{hyperref}

\hypersetup{colorlinks=true,linkcolor=blue,citecolor=blue,urlcolor=blue}

\DeclareMathOperator{\gp}{gp}
\newcommand{\stp}{\mathbin{\boxtimes}}

\newtheorem{theorem}{Theorem}[section]
\newtheorem{lemma}[theorem]{Lemma}
\newtheorem{corollary}[theorem]{Corollary}
\newtheorem{proposition}[theorem]{Proposition}
\newtheorem{fact}[theorem]{Fact}

\newtheorem{problem}[theorem]{Problem}

\theoremstyle{definition}

\newtheorem{remark}[theorem]{Remark}

\begin{document}

\title{General position sets in strong products with paths and cycles}
\author{Aleksander Vesel\\
\small Faculty of Natural Sciences and Mathematics, University of Maribor, Slovenia\\
\small Institute of Mathematics, Physics and Mechanics, Ljubljana, Slovenia\\
\texttt{aleksander.vesel@um.si}\\
\small ORCID: 0000-0003-3705-0071}
\date{}
\maketitle

\begin{abstract}
We study general position sets in strong products involving paths and cycles. For every connected graph $H$ and every $s\ge 2$, we prove that
$\gp(P_s\stp H)=2\gp(H)$. We also determine the corresponding values when the path is replaced by $C_4$, $C_5$, or $C_6$, and establish a general upper bound for $\gp(C_s\stp H)$. These results are then applied to strong products of two cycles. We determine several exact values, construct infinite families attaining the general upper bound, and provide counterexamples to the conjectured multiplicativity of the general position number under the strong product.
\end{abstract}



\section{Introduction}\label{sec:intro}

Let $G=(V(G),E(G))$ be a graph. A set $S\subseteq V(G)$ is a \emph{general position set} if no vertex of $S$ lies on a shortest path between two other vertices of $S$. Equivalently, for every three pairwise distinct vertices $u,v,w\in S$,
\[
d_G(u,v)\neq d_G(u,w)+d_G(w,v),
\]
where $d_G(u,v)$, or simply $d(u,v)$ when the graph is clear, denotes the distance between $u$ and $v$ in $G$. The maximum cardinality of a general position set in $G$ is the \emph{general position number} of $G$, denoted by $\gp(G)$.

General position sets were introduced independently in \cite{Chandran} and \cite{manuel-2018}. Since then, they have attracted considerable attention, particularly in the context of graph operations; see, for example, \cite{mpproduct,ghorbani-2021,polynomial, KlavzarRus,KorzeVesel, Manuel-2018b,tian-2024+,tian-2021a,tian-2021b}. Other notable contributions include \cite{AnaChaChaKlaTho,patkos-2020,yao-2022}. For a comprehensive overview of the theory of general position sets and related developments, we refer the reader to the survey \cite{survey}.

General position sets in strong products were investigated in \cite{KlavzarY}. Among the problems posed there was whether the general position number of the strong product of two connected graphs always equals the product of the general position numbers of the factors. We answer this question in the negative.

In this paper, we study strong products in which at least one factor is a path or a cycle. The remainder of this section introduces the required terminology and recalls the results used later. In Section~\ref{sec:path}, we determine the general position number of the strong product of an arbitrary connected graph with a path. Section~\ref{sec:cycle} concerns strong products with a cycle: we obtain exact values for several small cycles and a general upper bound. In Section~\ref{twocycles}, we focus on strong products of two cycles, determine exact values and bounds for several families, and show that the multiplicativity question from \cite{KlavzarY} has a negative answer.

All graphs considered in this paper are finite, simple, and connected. We use the well-known values
\[
\gp(P_n)=2 \quad (n\ge 2),
\]
and
\[
\gp(C_n)=
\begin{cases}
2, & n=4,\\
3, & n=3\text{ or }n\ge 5.
\end{cases}
\]
Indeed, every three vertices of $C_4$ contain an antipodal pair, and each of the remaining two vertices lies on a shortest path between the antipodal vertices.

Let $G=(V(G),E(G))$ and $H=(V(H),E(H))$. The \emph{strong product} $G\stp H$ has vertex set $V(G)\times V(H)$, and distinct vertices $(u_x,u_y)$ and $(v_x,v_y)$ are adjacent if and only if
\begin{itemize}
\item $u_x=v_x$ and $u_yv_y\in E(H)$, or
\item $u_y=v_y$ and $u_xv_x\in E(G)$, or
\item $u_xv_x\in E(G)$ and $u_yv_y\in E(H)$.
\end{itemize}
For $u=(u_x,u_y)\in V(G\stp H)$, the vertices $u_x$ and $u_y$ are called the \emph{$x$-projection} and the \emph{$y$-projection} of $u$, respectively.

For $p\in V(H)$, the subgraph induced by $V(G)\times\{p\}$ is a copy of $G$, denoted by $G^p$. Similarly, for $q\in V(G)$, the subgraph induced by $\{q\}\times V(H)$ is a copy of $H$.

A subgraph $H'$ of a graph $G$ is \emph{isometric} if
$d_{H'}(u,v)=d_G(u,v)$ for all $u,v\in V(H')$.

For a positive integer $k$, let
\[
[k]=\{1,2,\ldots,k\},
\qquad
[k]_0=\{0,1,\ldots,k-1\}.
\]
For $s\ge 3$ and $n,k\in[s]_0$, define the cyclic distance
\[
|n-k|_s=\min\{|n-k|,s-|n-k|\}.
\]
We identify $V(P_s)$ with $[s]_0$ for $s\ge 2$ and $V(C_s)$ with $[s]_0$ for $s\ge 3$. Thus, a vertex of $C_s\stp C_t$ is written as $u=(u_x,u_y)$, where $u_x\in[s]_0$ and $u_y\in[t]_0$.

The following result is well known; see, for example, \cite{imkl-00}.

\begin{fact}\label{distances}
If $u=(u_x,u_y)$ and $v=(v_x,v_y)$ are vertices of $G\stp H$, then
\[
d_{G\stp H}(u,v)=\max\{d_G(u_x,v_x),d_H(u_y,v_y)\}.
\]
\end{fact}

For a vertex $u$ of $G$, let $N_G(u)$ denote its open neighborhood and let
$N_G[u]=N_G(u)\cup\{u\}$.

The \emph{interval} $I_G(u,v)$ between vertices $u$ and $v$ of $G$ is the set of all vertices that lie on a shortest $u,v$-path. When the graph is clear, we write $I(u,v)$. 

Thus, a set $S\subseteq V(G)$ is in general position if and only if
$z\notin I_G(u,v)$ for every three pairwise distinct vertices $u,v,z\in S$.

Let $S$ be a general position set of $G\stp H$, and let $A\subseteq S$. We call $A$ an \emph{$x$-isometric set} if
\[
d_G(u_x,v_x)\ge d_H(u_y,v_y)
\]
for all $u,v\in A$. Analogously, $A$ is a \emph{$y$-isometric set} if
\[
d_H(u_y,v_y)\ge d_G(u_x,v_x)
\]
for all $u,v\in A$.

\begin{proposition}\label{interval}
Let $G$ and $H$ be graphs, and let $u,v,z\in V(G\stp H)$. Suppose that
$d_G(u_x,v_x)\ge d_H(u_y,v_y)$. Then $z\in I_{G\stp H}(u,v)$ if and only if
\begin{enumerate}
\item $d_G(u_x,z_x)\ge d_H(u_y,z_y)$;
\item $d_G(z_x,v_x)\ge d_H(z_y,v_y)$;
\item $z_x\in I_G(u_x,v_x)$.
\end{enumerate}
\end{proposition}

\begin{proof}
Suppose first that $z\in I_{G\stp H}(u,v)$. By Fact~\ref{distances} and the hypothesis,
\begin{align*}
d_G(u_x,v_x)
&=d_{G\stp H}(u,v)\\
&=d_{G\stp H}(u,z)+d_{G\stp H}(z,v)\\
&\ge d_G(u_x,z_x)+d_G(z_x,v_x)\\
&\ge d_G(u_x,v_x).
\end{align*}
Hence equality holds throughout. In particular,
\[
d_G(u_x,v_x)=d_G(u_x,z_x)+d_G(z_x,v_x),
\]
so $z_x\in I_G(u_x,v_x)$. Moreover,
\[
d_{G\stp H}(u,z)=d_G(u_x,z_x)
\quad\text{and}\quad
d_{G\stp H}(z,v)=d_G(z_x,v_x),
\]
which, by Fact~\ref{distances}, yields conditions (1) and (2).

Conversely, suppose that conditions (1)--(3) hold. Condition (3) gives
\[
d_G(u_x,v_x)=d_G(u_x,z_x)+d_G(z_x,v_x).
\]
By conditions (1) and (2) and Fact~\ref{distances},
\[
d_{G\stp H}(u,z)=d_G(u_x,z_x)
\quad\text{and}\quad
d_{G\stp H}(z,v)=d_G(z_x,v_x).
\]
Therefore,
\begin{align*}
d_{G\stp H}(u,z)+d_{G\stp H}(z,v)
&=d_G(u_x,v_x)\\
&=d_{G\stp H}(u,v),
\end{align*}
and hence $z\in I_{G\stp H}(u,v)$.
\end{proof}

Klav\v{z}ar and Yero \cite{KlavzarY} established several exact results for strong products. In particular, they proved the following general lower bound.
\begin{theorem}[{\cite[Theorem~4.2]{KlavzarY}}]\label{th:lower}
If $G$ and $H$ are graphs, then
\[
\gp(G\stp H)\ge \gp(G)\gp(H).
\]
\end{theorem}
They also asked whether equality always holds.
\begin{problem}[{\cite[Problem~4.8]{KlavzarY}}]\label{conjecture}
Is
\[
\gp(G\stp H)=\gp(G)\gp(H)
\]
true for all connected graphs $G$ and $H$?
\end{problem}

\begin{fact}\label{fact:iso-bound}
Let $S$ be a general position set of $G\stp H$ and let $A\subseteq S$.
If $A$ is $y$-isometric, then $|A|\le \gp(H)$. Analogously, if $A$ is
$x$-isometric, then $|A|\le \gp(G)$.
\end{fact}

\begin{proof}
Assume that $A$ is $y$-isometric. Distinct vertices of $A$ have distinct
$y$-projections: if $u\ne v$ and $u_y=v_y$, then
\[
0=d_H(u_y,v_y)\ge d_G(u_x,v_x),
\]
which implies $u=v$, a contradiction. Let $u,v,z\in A$ be pairwise distinct.
Because $A$ is contained in a general position set,
\[
d_{G\stp H}(u,v)\ne d_{G\stp H}(u,z)+d_{G\stp H}(z,v).
\]
By Fact~\ref{distances} and the $y$-isometric property,
$d_{G\stp H}(a,b)=d_H(a_y,b_y)$ for all $a,b\in A$. Consequently,
\[
d_H(u_y,v_y)\ne d_H(u_y,z_y)+d_H(z_y,v_y).
\]
Thus, the $y$-projections of the vertices of $A$ form a general position set in
$H$, and hence $|A|\le \gp(H)$. The $x$-isometric case is analogous.
\end{proof}

For $u,v\in V(G\stp H)$, we use the notation
\[
d_x(u,v)=d_G(u_x,v_x),
\qquad
d_y(u,v)=d_H(u_y,v_y).
\]
When the first factor is $P_s$, we have $d_x(u,v)=|u_x-v_x|$, whereas for
$C_s$ we have $d_x(u,v)=|u_x-v_x|_s$.

Let $H$ be a connected graph, let $G\in\{P_s,C_s\}$, and let $S$ be a general
position set of $G\stp H$. Define the graph $X=X(S)$ by
$V(X)=S$ and
\[
uv\in E(X)
%\quad\Longleftrightarrow\quad
\text{ if and only if }
d_x(u,v)>d_y(u,v).
\]
Thus, pairs satisfying $d_x(u,v)=d_y(u,v)$ are nonadjacent in $X$. 

A set
$A\subseteq S$ is independent in $X$ if and only if it is $y$-isometric. By
Fact~\ref{fact:iso-bound}, every such set has cardinality at most $\gp(H)$.

\begin{fact}\label{fact:chromatic}
Let $H$ be a connected graph, let $G\in\{P_s,C_s\}$, and let $S$ be a general
position set of $G\stp H$. Then
\[
|S|\le \chi(X(S))\gp(H).
\]
\end{fact}

\begin{proof}
A proper coloring of $X(S)$ partitions $S$ into $\chi(X(S))$ independent sets.
Each color class has cardinality at most $\gp(H)$ by
Fact~\ref{fact:iso-bound}, and the result follows.
\end{proof}

If $uv\in E(X)$, we also write $u\sim_X v$.

\section{Strong products of a graph and a path}\label{sec:path}

Throughout this section let $S$ be a general position set of $P_s \stp H$,
where $s\ge 2$ and $H$ is a connected graph. 
\begin{lemma}\label{lem:bipartite}
The graph $X(S)$ is bipartite.
\end{lemma}

\begin{proof}
Suppose that $X=X(S)$ contains an odd cycle, and choose one of minimum length.
This cycle is induced. It cannot be a triangle, because the vertices of a
triangle in $X$ form an $x$-isometric set, whose cardinality is at most
$\gp(P_s)=2$ by Fact~\ref{fact:iso-bound}.

Hence, $X$ contains an induced odd cycle
\[
w_0w_1\cdots w_{2k}w_0,
\qquad k\ge 2.
\]
Write $p_i=(w_i)_x$, let
$\delta_{i,j}=d_y(w_i,w_j)$, and set
$s_i=p_{i+1}-p_i$, where indices are taken modulo $2k+1$. Since
$w_iw_{i+1}\in E(X)$ and $w_iw_{i+2}\notin E(X)$, we have
\[
|s_i|\ge \delta_{i,i+1}+1
\quad\text{and}\quad
|s_i+s_{i+1}|\le \delta_{i,i+2}.
\]
If $s_i$ and $s_{i+1}$ had the same sign, then the triangle inequality in $H$
would yield
\begin{align*}
|s_i+s_{i+1}|
&=|s_i|+|s_{i+1}|\\
&\ge \delta_{i,i+1}+\delta_{i+1,i+2}+2\\
&>\delta_{i,i+2}\\
&\ge |s_i+s_{i+1}|,
\end{align*}
a contradiction. Therefore, consecutive steps alternate in sign. This is
impossible around a cycle with an odd number of vertices.
\end{proof}

\begin{theorem}\label{thm:path}
If $s\ge 2$ and $H$ is a connected graph, then
\[
\gp(P_s\stp H)=2\gp(H).
\]
\end{theorem}

\begin{proof}
Theorem~\ref{th:lower} gives
\[
\gp(P_s\stp H)\ge \gp(P_s)\gp(H)=2\gp(H).
\]
Conversely, let $S$ be a general position set of $P_s\stp H$. By
Lemma~\ref{lem:bipartite}, $X(S)$ is bipartite, and hence
$\chi(X(S))\le 2$. Fact~\ref{fact:chromatic} now gives
$|S|\le 2\gp(H)$.
\end{proof}

\section{Strong products of a graph and a cycle}\label{sec:cycle}

\begin{theorem}\label{thm:averaging}
Let $H$ be a connected graph. If $s\ge 3$ and $h=\lfloor s/2\rfloor$, then
\[
\gp(C_s\stp H)\le
\left\lfloor\frac{2s\,\gp(H)}{h+1}\right\rfloor.
\]
\end{theorem}

\begin{proof}
For $i\in[s]_0$, let
\[
A_i=\{i,i+1,\ldots,i+h\},
\]
where the entries are taken modulo $s$. The subgraph of $C_s$ induced by $A_i$
is an isometric path on $h+1$ vertices. Hence, by Fact~\ref{distances}, the
subgraph induced by $A_i\times V(H)$ is an isometric copy of
$P_{h+1}\stp H$ in $C_s\stp H$.

Let $S$ be a general position set of $C_s\stp H$ and set
$S_i=S\cap(A_i\times V(H))$. Since $A_i\times V(H)$ induces an isometric
subgraph, $S_i$ is a general position set of $P_{h+1}\stp H$. By
Theorem~\ref{thm:path},
\[
|S_i|\le 2\gp(H).
\]
Every vertex of $S$ belongs to exactly $h+1$ of the sets
$A_i\times V(H)$. Therefore,
\[
|S|(h+1)=\sum_{i=0}^{s-1}|S_i|\le 2s\gp(H).
\]
Since $|S|$ is an integer, the asserted bound follows.
\end{proof}

Theorem~\ref{thm:averaging} will be applied to strong products of two cycles in Section~\ref{twocycles}. When one factor is a $3$-cycle, we use the following result.

\begin{proposition}[{\cite[Proposition~4.3]{KlavzarY}}]\label{prop:complete}
For every $m\ge 1$ and every connected graph $H$,
\[
\gp(K_m\stp H)=m\gp(H).
\]
\end{proposition}

\begin{remark}
The concepts introduced above also yield a short proof of
Proposition~\ref{prop:complete}. The lower bound follows from
Theorem~\ref{th:lower}. For the upper bound, let $S$ be a general position set
of $K_m\stp H$, and choose a set $\widehat S\subseteq S$ containing one vertex
from each nonempty layer $K_m^p$ met by $S$. Distinct vertices
$u,v\in\widehat S$ lie in distinct layers, and therefore
$d_y(u,v)\ge 1$, while $d_x(u,v)\le 1$. Thus, $\widehat S$ is
$y$-isometric, and Fact~\ref{fact:iso-bound} gives
$|\widehat S|\le \gp(H)$. Hence, $S$ meets at most $\gp(H)$ layers, each of
which contains at most $m$ vertices, so $|S|\le m\gp(H)$.
\end{remark}

Let $H$ be a connected graph, let $s\ge 4$, let $S$ be a general position set
of $C_s\stp H$, and set $X=X(S)$. If $uv\in E(X)$ and $d_x(u,v)=1$, then
$u$ and $v$ are called \emph{close neighbors}, and $uv$ is called a
\emph{short edge} of $X$. Vertices $u,v\in V(X)$ are \emph{adjacent true twins}
if $N_X[u]=N_X[v]$.

\begin{lemma}\label{lem:short}
Let $H$ be a connected graph, let $s\ge 4$, let $S$ be a general position set
of $C_s\stp H$, and let $X=X(S)$. If $uv$ is a short edge of $X$, then
$u_y=v_y$, and $u$ and $v$ are adjacent true twins of $X$. Moreover, the
short edges form a matching, and no edge of $X$ joins endpoints of two distinct
short edges.
\end{lemma}

\begin{proof}
Since $d_x(u,v)=1>d_y(u,v)$, we have $d_y(u,v)=0$, and hence $u_y=v_y$.
Let $w\sim_X v$, where $w\notin\{u,v\}$, and suppose that
$w\not\sim_X u$. Put $q=d_x(v,w)$. Then
$d_y(v,w)\le q-1$, and because $u_y=v_y$,
$d_y(u,w)=d_y(v,w)$. The triangle inequality in $C_s$ gives
\[
q-1\le d_x(u,w)\le d_y(u,w)\le q-1.
\]
Thus, equality holds throughout. By Fact~\ref{distances},
\[
d(u,w)=q-1,
\qquad d(u,v)=1,
\qquad d(v,w)=q.
\]
Consequently, $u\in I(v,w)$, contradicting the fact that $S$ is in general
position. Interchanging $u$ and $v$ shows that $N_X[u]=N_X[v]$.

If two short edges shared a vertex, their three endpoints would have the same
$y$-projection and consecutive $x$-projections. Their pairwise product
distances would therefore be $1,1,2$, contradicting general position. Hence,
the short edges form a matching.

Finally, suppose that an edge joins endpoints of two distinct short edges.
Since the endpoints of each short edge are adjacent true twins, the four
endpoints induce a $K_4$ in $X$. They therefore form an $x$-isometric set of
cardinality $4$, contradicting Fact~\ref{fact:iso-bound} and
$\gp(C_s)\le 3$.
\end{proof}

\begin{theorem}\label{thm:rows456}
For every connected graph $H$,
\begin{align*}
\gp(C_4\stp H)&=2\gp(H),\\
\gp(C_5\stp H)&=\gp(C_6\stp H)=3\gp(H).
\end{align*}
\end{theorem}

\begin{proof}
The lower bounds follow from Theorem~\ref{th:lower}, together with
$\gp(C_4)=2$ and $\gp(C_5)=\gp(C_6)=3$. For the upper bounds, let $S$ be a
general position set and let $X=X(S)$. By Fact~\ref{fact:chromatic}, it is
enough to prove that $\chi(X)\le 2$ when $s=4$ and that $\chi(X)\le 3$ when
$s\in\{5,6\}$.

First suppose that $s\in\{4,6\}$. We claim that every pair of close neighbors
forms an isolated component of $X$. Let $u$ and $v$ be close neighbors and
suppose that $w$ is a common external neighbor, it is adjacent to both $u$ and
$v$ by Lemma~\ref{lem:short}. Up to order, the distances in $C_s$ from $w_x$
to $u_x$ and $v_x$ are $(1,2)$ when $s=4$, and either $(1,2)$ or $(2,3)$ when
$s=6$. If the smaller distance is $1$, the corresponding edge is short.
Lemma~\ref{lem:short} then implies that $u$, $v$, and $w$ have the same
$y$-projection, and their product distances are $1,1,2$, a contradiction. If
the distances are $(2,3)$, then
$d_y(w,u)=d_y(w,v)\le 1$, and the product distances are $1,2,3$, again a
contradiction. Thus, no such external neighbor exists.

For $s=4$, every non-short edge has $x$-distance $2$. The two position classes
$\{0,1\}$ and $\{2,3\}$ therefore give a proper $2$-coloring of all non-short
edges, while each isolated short edge can be colored with the two colors.
Hence, $\chi(X)\le 2$.

For $s=6$, every non-short edge has $x$-distance $2$ or $3$. The three classes
$\{0,1\}$, $\{2,3\}$, and $\{4,5\}$ separate every such pair. Each isolated
short edge can be colored with any two of the three colors, and therefore
$\chi(X)\le 3$.

It remains to consider $s=5$. As in the $(1,2)$ case above, every external
neighbor of a pair of close neighbors has $x$-distance $2$ from both
endpoints. In $C_5$, there is a unique position with this property. Hence, all
external neighbors of a fixed pair have the same $x$-projection and are
pairwise nonadjacent in $X$. Remove the endpoints of all short edges. Every
remaining edge has $x$-distance $2$, so the position classes
$\{0,1\}$, $\{2,3\}$, and $\{4\}$ give a proper $3$-coloring of the remaining
vertices. Now restore the endpoints of the short edges. By
Lemma~\ref{lem:short}, distinct short edges are nonadjacent, while the external
neighbors of each short edge all have the same color. The two endpoints can
therefore be assigned the other two colors. Thus, $\chi(X)\le 3$.
\end{proof}

\section{Strong products of two cycles}\label{twocycles}
The following result is an immediate consequence of
Theorem~\ref{thm:averaging}.

\begin{proposition}\label{prop:sup}
Let $s,t\ge 7$. Then
\[
\gp(C_s\stp C_t)\le
\begin{cases}
9, & 8\in\{s,t\},\\
10, & 8\notin\{s,t\}\text{ and }
      \{s,t\}\cap\{7,9,10,12,14,16,18,20\}\ne\varnothing,\\
11, & \text{otherwise}.
\end{cases}
\]
\end{proposition}

\begin{proof}
By Theorem~\ref{thm:averaging},
\[
\gp(C_s\stp C_t)
\le
\min\left\{
\left\lfloor\frac{6s}{\lfloor s/2\rfloor+1}\right\rfloor,
\left\lfloor\frac{6t}{\lfloor t/2\rfloor+1}\right\rfloor
\right\}.
\]
For $r=2m+1$, the expression before taking the floor is
$12-6/(m+1)$, while for $r=2m$ it is $12-12/(m+1)$. For $r\ge 7$, the
corresponding floor equals $9$ when $r=8$, equals $10$ when
$r\in\{7,9,10,12,14,16,18,20\}$, and equals $11$ otherwise. Taking the
minimum proves the result.
\end{proof}

Proposition~\ref{prop:sup}, together with Theorems~\ref{th:lower} and~\ref{thm:rows456}, determines the general position number whenever one of the factors is a cycle of length~$8$. \begin{corollary}\label{cor:row8} 
If $t\ge 3$ and $t\ne 4$, then \[ \gp(C_8\stp C_t)=9. \] 
\end{corollary}

Theorem~\ref{thm:rows456} and Proposition~\ref{prop:complete} determine the
answer whenever at least one factor is a cycle of length at most $6$.

\begin{proposition}\label{cor:upper}
Let $s,t\ge 3$ and suppose that $\min\{s,t\}\le 6$. Then
\[
\gp(C_s\stp C_t)=
\begin{cases}
4, & s=t=4,\\
6, & \text{exactly one of }s,t\text{ equals }4,\\
9, & 4\notin\{s,t\}.
\end{cases}
\]
\end{proposition}

\begin{proof}
If one factor is $C_3=K_3$, the result follows from
Proposition~\ref{prop:complete}. If one factor is $C_4$, $C_5$, or $C_6$, it
follows from Theorem~\ref{thm:rows456}, together with the known values of the
general position number of a cycle.
\end{proof}

Upper bounds can be refined by using the next lemma in some cases.
\begin{lemma}\label{lem:perfect}
Let $s,t\ge 3$ with $s\ne 4$ and $t\ne 4$. Let $S$ be a general position set
of $C_s\stp C_t$, and let $X=X(S)$. Then $\omega(X)\le 3$ and
$\alpha(X)\le 3$. Moreover, if $X$ is perfect, then $|S|\le 9$.
\end{lemma}

\begin{proof}
A clique of $X$ is an $x$-isometric set, so
$\omega(X)\le \gp(C_s)=3$. An independent set of $X$ is $y$-isometric, so
$\alpha(X)\le \gp(C_t)=3$ by Fact~\ref{fact:iso-bound}. If $X$ is perfect,
then $\chi(X)=\omega(X)\le 3$, and Fact~\ref{fact:chromatic} gives
$|S|\le 3\gp(C_t)=9$.
\end{proof}

\begin{remark}
More generally, Fact~\ref{fact:chromatic} gives $|S|\le 9$ whenever
$\chi(X)\le 3$. Perfection is sufficient but not necessary. For example, in
$C_5\stp C_5$, the general position set
$S=\{(0,0),(1,2),(2,4),(3,1),(4,3)\}$ has all pairwise distances equal to
$2$, while $X(S)$ is an induced $C_5$ and is therefore imperfect. Nevertheless,
$\gp(C_5\stp C_5)=9$. Conversely, since $\alpha(X)\le 3$, every general
position set with at least $10$ vertices has $\chi(X)\ge 4$.
\end{remark}

We next show that, for cycles of length at least $7$, some strong products
$C_s\stp C_t$ have general position number greater than $9$. Consequently, the
question posed in \cite{KlavzarY} has a negative answer.

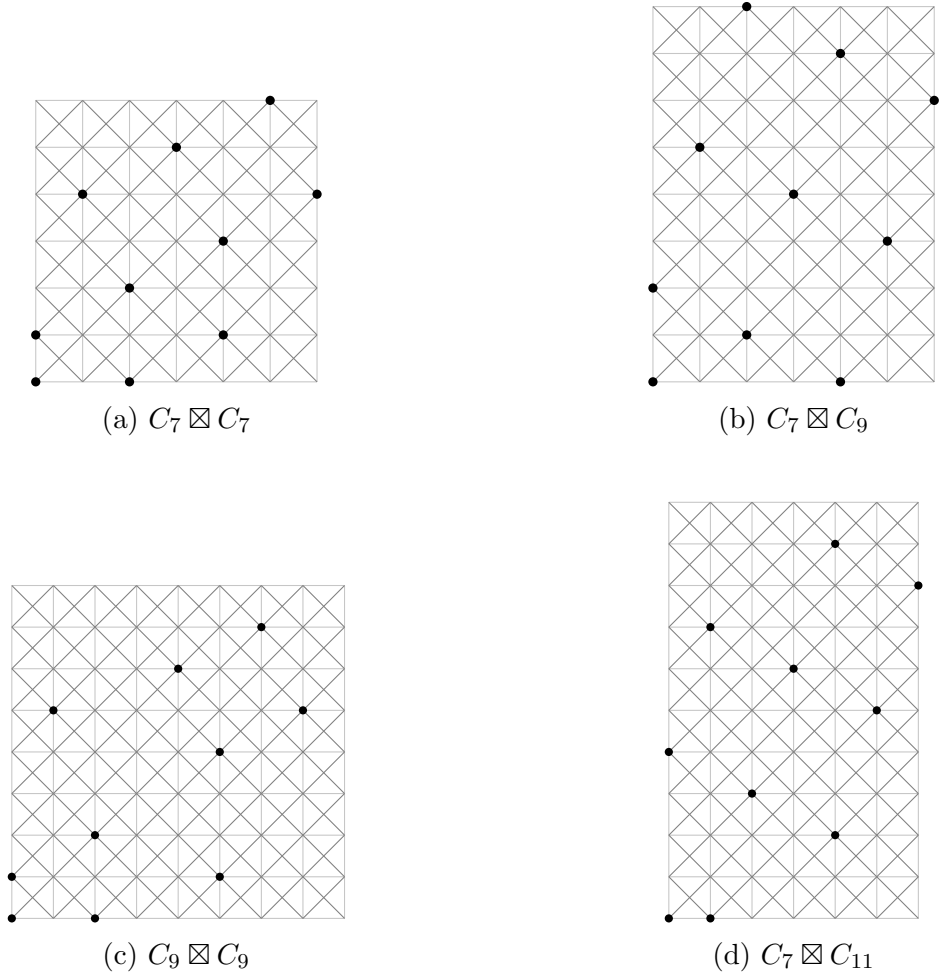
\begin{figure}[ht] 
\centering
\begin{subfigure}[t]{0.47\textwidth}
\centering
\begin{tikzpicture}[scale=0.62]
    \draw[step=1, gray!45, thin] (0,0) grid (6,6);

    \foreach \x in {0,...,5} {
    \foreach \y in {0,...,5} {
      % Draws diagonal lines from corner to corner of each 1cm x 1cm block
      \draw[gray, thin] (\x,\y) -- (\x+1,\y+1);
      \draw[gray, thin] (\x,\y+1) -- (\x+1,\y);
    }
  }

    \foreach \x/\y in {
        0/0,0/1,1/4,2/0,2/2,
        3/5,4/1,4/3,5/6,6/4
    }
        \fill (\x,\y) circle (2.8pt);
\end{tikzpicture}
\caption{$C_7 \stp C_7$}
\end{subfigure}
\hfill
% C7 x C9
\begin{subfigure}[t]{0.47\textwidth}
\centering
\begin{tikzpicture}[scale=0.62]
    \draw[step=1, gray!45, thin] (0,0) grid (6,8);

    \foreach \x in {0,...,5} {
    \foreach \y in {0,...,7} {
      % Draws diagonal lines from corner to corner of each 1cm x 1cm block
      \draw[gray, thin] (\x,\y) -- (\x+1,\y+1);
      \draw[gray, thin] (\x,\y+1) -- (\x+1,\y);
    }
  }

    \foreach \x/\y in {
        0/0,0/2,1/5,2/1,2/8,
        3/4,4/0,4/7,5/3,6/6
    }
        \fill (\x,\y) circle (2.8pt);
\end{tikzpicture}
\caption{$C_7 \stp C_9$}
\end{subfigure}

\vspace{0.8cm}

% C9 x C9
\begin{subfigure}[t]{0.47\textwidth}
\centering
\begin{tikzpicture}[scale=0.55]
    \draw[step=1, gray!45, thin] (0,0) grid (8,8);

    \foreach \x in {0,...,7} {
    \foreach \y in {0,...,7} {
      % Draws diagonal lines from corner to corner of each 1cm x 1cm block
      \draw[gray, thin] (\x,\y) -- (\x+1,\y+1);
      \draw[gray, thin] (\x,\y+1) -- (\x+1,\y);
    }
  }

    \foreach \x/\y in {
        0/0,0/1,1/5,2/0,2/2,
        4/6,5/1,5/4,6/7,7/5
    }
        \fill (\x,\y) circle (2.8pt);
\end{tikzpicture}
\caption{$C_9 \stp C_9$}
\end{subfigure}
\hfill
% C7 x C11
\begin{subfigure}[t]{0.47\textwidth}
\centering
\begin{tikzpicture}[scale=0.55]
    \draw[step=1, gray!45, thin] (0,0) grid (6,10);

    \foreach \x in {0,...,5} {
    \foreach \y in {0,...,9} {
      % Draws diagonal lines from corner to corner of each 1cm x 1cm block
      \draw[gray, thin] (\x,\y) -- (\x+1,\y+1);
      \draw[gray, thin] (\x,\y+1) -- (\x+1,\y);
    }
  }

    \foreach \x/\y in {
        0/0,0/4,1/0,1/7,2/3,
        3/6,4/2,4/9,5/5,6/8
    }
        \fill (\x,\y) circle (2.8pt);
\end{tikzpicture}
\caption{$C_7 \stp C_{11}$} 
\end{subfigure}

\caption{General position sets of cardinality $10$ in some strong products of cycles.}
\label{fig:counterexamples}
\end{figure}

\begin{proposition}\label{prop:counterexamples}
Up to interchanging $s$ and $t$, the following exact values hold:
\[
\gp(C_s\stp C_t)=
\begin{cases}
10, & (s,t)\in\{(7,7),(7,9),(7,11),(9,9),(10,10)\},\\
11, & s=11\text{ and }t\in\{11,13,15,17,19,21,22\}.
\end{cases}
\]
\end{proposition}

\begin{proof}
The four general position sets of cardinality $10$ displayed in
Figure~\ref{fig:counterexamples} establish the lower bounds for
$(s,t)\in\{(7,7),(7,9),(7,11),(9,9)\}$. A general position set of cardinality
$10$ in $C_{10}\stp C_{10}$ is
\[
\{(i,-3i\bmod 10):i\in[10]_0\}.
\]
Similarly,
\[
\{(i,3i\bmod 11):i\in[11]_0\}
\]
is a general position set of cardinality $11$ in
$C_{11}\stp C_{11}$.

For the remaining values of $t$, the following sets $S_t$ have cardinality $11$ and
are in general position in $C_{11}\stp C_t$:
\begin{align*}
S_{13}={}&\{(0,0),(1,9),(2,4),(3,12),(4,8),(5,3),\\
          &\hspace{2.1cm}(6,11),(7,6),(8,2),(9,10),(10,5)\},\\[1mm]
S_{15}={}&\{(0,0),(1,5),(2,10),(3,1),(4,6),(5,12),\\
          &\hspace{2.1cm}(6,2),(7,8),(8,13),(9,4),(10,9)\},\\[1mm]
S_{17}={}&\{(0,0),(1,6),(2,12),(3,1),(4,7),(5,13),\\
          &\hspace{2.1cm}(6,3),(7,9),(8,15),(9,4),(10,10)\},\\[1mm]
S_{19}={}&\{(0,0),(1,6),(2,13),(3,1),(4,8),(5,15),\\
          &\hspace{2.1cm}(6,3),(7,10),(8,17),(9,5),(10,12)\},\\[1mm]
S_{21}={}&\{(0,0),(1,7),(2,15),(3,1),(4,9),(5,17),\\
          &\hspace{2.1cm}(6,3),(7,11),(8,19),(9,5),(10,13)\},\\[1mm]
S_{22}={}&\{(0,0),(1,8),(2,16),(3,2),(4,10),(5,18),\\
          &\hspace{2.1cm}(6,4),(7,12),(8,20),(9,6),(10,14)\}.
\end{align*}
A direct check using the distance formula from Fact~\ref{distances} verifies
that each displayed set is in general position. The matching upper bounds are
given by Proposition~\ref{prop:sup}.
\end{proof}

If one cycle is sufficiently longer than the other, then the general position number of their strong product is equal to 9, as we show in the sequel.

\begin{lemma}\label{lem:cluster}
Let $S\subseteq V(C_s)$ and suppose that $d_{C_s}(u,v)\le r$ for all
$u,v\in S$. If $s>3r$, then $S$ is contained in an induced path of length at
most $r$ in $C_s$.
\end{lemma}

\begin{proof}
Fix $p\in S$. Since $d_{C_s}(p,u)\le r<s/2$ for every $u\in S$, each vertex
of $S$ lies uniquely either clockwise or counterclockwise from $p$ at distance
at most $r$. Let $r_+$ and $r_-$ be the largest distances occurring in the two
directions, taking the corresponding value to be $0$ if no vertex lies in that
direction. We claim that $r_++r_-\le r$.

Suppose otherwise. The two corresponding extreme vertices have cyclic distance
at most $r$. Since $r_++r_->r$, the shorter route between them would have to be
the complementary arc, and hence
\[
s-(r_++r_-)\le r.
\]
It follows that $r_++r_-\ge s-r>2r$, contradicting
$r_+,r_-\le r$. Thus, $r_++r_-\le r$, and the arc between the two extreme
vertices that contains $p$ is an induced path of length at most $r$ containing
$S$.
\end{proof}

\begin{theorem}\label{thm:window}
Let $s,t\ge 3$ with $s\ne 4$ and $t\ne 4$. If
$t>15\bigl(\lfloor s/2\rfloor-1\bigr)$, 
then
\[
\gp(C_s\stp C_t)=9.
\]
\end{theorem}

\begin{proof}
The lower bound follows from Theorem~\ref{th:lower}, because
$\gp(C_s)=\gp(C_t)=3$. Suppose, for a contradiction, that there exists a
general position set with at least $10$ vertices, and let $S$ be a
$10$-vertex subset of it. Set
\[
r=\lfloor s/2\rfloor-1
\]
and let $X=X(S)$. Every edge of $X$ satisfies $d_y\le r$.

Every component of $X$ has diameter at most $5$. Indeed, a shortest path with
six edges is isometric, and its four alternating vertices form an independent
set, contradicting $\alpha(X)\le 3$ from Lemma~\ref{lem:perfect}. Consequently,
the $y$-projections of the vertices in any component have pairwise distance at
most $5r$.

Since $t>15r$, Lemma~\ref{lem:cluster} implies that the $y$-projections of each
component $C$ lie on an induced isometric path in $C_t$. If necessary, extend this
path by one vertex, so that it has at least two vertices. The vertices of the
component then lie in an isometric subgraph of $C_s\stp C_t$ isomorphic to
$C_s\stp P_m$ for some $m\ge 2$. By commutativity of the strong product and
Theorem~\ref{thm:path}, $C$ has at most
\[
2\gp(C_s)=6
\]
vertices.

By Lemma~\ref{lem:perfect}, $X$ is $K_4$-free. A $K_4$-free graph on
$q\le 6$ vertices has independence number at least $\lceil q/3\rceil$.
Independence numbers add over components, so
\[
\alpha(X)\ge \sum_C\left\lceil\frac{|V(C)|}{3}\right\rceil
\ge \left\lceil\frac{10}{3}\right\rceil=4,
\]
again contradicting Lemma~\ref{lem:perfect}. Therefore,
$\gp(C_s\stp C_t)\le 9$.
\end{proof}

\begin{theorem}\label{thm:row7}
If $t\ge 3$ and $t\ne 4$, then
\[
\gp(C_7\stp C_t)=
\begin{cases}
10, & t\in\{7,9,11\},\\
9, & \text{otherwise}.
\end{cases}
\]
\end{theorem}

\begin{proof}
For $t>30$, the value $9$ follows from Theorem~\ref{thm:window}. The cases
$t\in\{3,5,6\}$ follow from Proposition~\ref{cor:upper}, and the case $t=8$
follows from Proposition~\ref{prop:sup} together with Theorem~\ref{th:lower}.
For $t\in\{7,9,11\}$, Proposition~\ref{prop:counterexamples} gives the value
$10$. The remaining cases
\[
t\in\{10,12,13,\ldots,30\}
\]
were verified by exhaustive computer search.
\end{proof}

We next show that the upper bound in Proposition~\ref{prop:sup} is attained for
infinitely many pairs of cycles.

\begin{lemma}\label{lem:scaling}
If $k\ge 1$, then
\[
\gp(C_{ks}\stp C_{kt})\ge \gp(C_s\stp C_t).
\]
\end{lemma}

\begin{proof}
The map $(x,y)\mapsto(kx,ky)$ satisfies
\[
d_{C_{ks}}(kx,kx')=k\,d_{C_s}(x,x')
\]
and analogously in the second coordinate. By Fact~\ref{distances}, it
multiplies all distances in the strong product by $k$. Hence, it preserves all
strict triangle inequalities, and the image of a general position set in
$C_s\stp C_t$ is a general position set in $C_{ks}\stp C_{kt}$.
\end{proof}

\begin{corollary}\label{cor:scaled11}
If $k\ge 1$, then
\[
\gp(C_{11k}\stp C_{11k})=11.
\]
\end{corollary}

\begin{proof}
The lower bound follows from Lemma~\ref{lem:scaling} and
Proposition~\ref{prop:counterexamples}. The upper bound follows from
Proposition~\ref{prop:sup}.
\end{proof}

The upper bound can also be attained when the two cycles have different lengths.

\begin{lemma}\label{lem:thick}
Let $S$ be a general position set of $C_s\stp C_t$. If $k\ge 1$ and
$a,b\ge 0$ satisfy $\max\{a,b\}<k$, then
\[
S'=\{(ku_x,ku_y):(u_x,u_y)\in S\}
\]
is a general position set of $C_{ks+a}\stp C_{kt+b}$.
\end{lemma}

\begin{proof}
For $x,x'\in V(C_s)$, let $t=|x-x'|$. Then
\begin{align*}
d_{C_{ks+a}}(kx,kx')
&=\min\{kt,k(s-t)+a\},\\
k\,d_{C_s}(x,x')
&=\min\{kt,k(s-t)\}.
\end{align*}
Therefore,
\[
k\,d_{C_s}(x,x')
\le d_{C_{ks+a}}(kx,kx')
\le k\,d_{C_s}(x,x')+a.
\]
The analogous inequality holds for the second coordinate. Let $d$ and $d'$ denote the distance functions in $C_s\stp C_t$ and
$C_{ks+a}\stp C_{kt+b}$, respectively. By Fact~\ref{distances}, for all
$u,v\in S$ and their images $u',v'\in S'$,
\[
k\,d(u,v)\le d'(u',v')\le k\,d(u,v)+M,
\]
where $M=\max\{a,b\}<k$.

Let $u,v,z\in S$ be pairwise distinct. Since $S$ is in general position,
there is an integer $\Delta\ge 1$ such that
\[
d(u,v)+d(v,z)=d(u,z)+\Delta.
\]
Using the preceding bounds, we obtain
\begin{align*}
d'(u',v')+d'(v',z')
&\ge k\bigl(d(u,v)+d(v,z)\bigr)\\
&=k\,d(u,z)+k\Delta\\
&>k\,d(u,z)+M\\
&\ge d'(u',z').
\end{align*}
Thus, $v'$ does not lie on a shortest $u',z'$-path. Since the triple was
arbitrary, $S'$ is in general position.
\end{proof}

\begin{corollary}\label{cor:cone} 
Let  $Q=\{11,13,15,17,19,21,22\}$. If $q\in Q$, $k\ge 1$, and $0\le a,b<k$, then \[ \gp(C_{11k+a}\stp C_{qk+b})=11. \] In particular, \[ \gp(C_s\stp C_s)=11 \qquad\text{for every }s\ge 121. \] 
\end{corollary} 
\begin{proof} By Proposition~\ref{prop:counterexamples}, for every $q\in Q$ there exists a general position set of cardinality $11$ in $C_{11}\stp C_q$. Since $0\le a,b<k$, Lemma~\ref{lem:thick} maps this set to a general position set of cardinality $11$ in \[ C_{11k+a}\stp C_{qk+b}. \] Consequently, \[ \gp(C_{11k+a}\stp C_{qk+b})\ge 11. \] 

Since Proposition~\ref{prop:sup} gives \[ \gp(C_{11k+a}\stp C_{qk+b})\le 11, \]  equality follows. For the final assertion, write \[ s=11k+a, \qquad k=\left\lfloor\frac{s}{11}\right\rfloor, \qquad 0\le a\le 10. \] If $s\ge121$, then $k\ge11$, and therefore $a<k$. Taking $q=11$ and $b=a$ in the first part yields \[ \gp(C_s\stp C_s)=11. \] \end{proof}

\section*{Acknowledgments}
This work was supported by the Slovenian Research and Innovation Agency (ARIS) through research core funding P1-0297 and projects J1-70016, J1-70045, and N1-0431.

\end{document}